\documentclass[11pt,reqno]{amsart}
\usepackage{graphicx}
\usepackage[margin=1.30in]{geometry}
\usepackage{amsmath}
\usepackage{amssymb}
\usepackage{amsthm}
\usepackage{amsfonts}
\usepackage{mathrsfs}
\usepackage{enumerate}
\usepackage{enumitem}
\usepackage[scriptsize,hang,raggedright]{subfigure}
\usepackage[cmyk]{xcolor}
\usepackage{mathtools}

\usepackage{hyperref}
\usepackage{pdfsync}
\usepackage{dsfont}
\usepackage{color}

\usepackage{cite}
\usepackage{bbm}
\usepackage{multirow}

\numberwithin{equation}{section} 

\newtheorem{theorem}{Theorem}
\newtheorem{proposition}[theorem]{Proposition}
\newtheorem{lemma}[theorem]{Lemma}

\theoremstyle{remark}
\newtheorem{remark}[theorem]{Remark}
\theoremstyle{definition}
\newtheorem{example}[theorem]{Example}
\newcommand{\R}{\mathbb{R}}
\newcommand{\disp}{\displaystyle}

\newcommand{\ba}{\begin{array}}
\newcommand{\ea}{\end{array}}

\newcommand{\tld}[1]{\widetilde{#1}}
\newcommand{\bthm}{\begin{theorem}}
\newcommand{\ethm}{\end{theorem}}
\newcommand{\bprop}{\begin{proposition}}
\newcommand{\eprop}{\end{proposition}}
\newcommand{\blemma}{\begin{lemma}}
\newcommand{\elemma}{\end{lemma}}
\newcommand{\bexmpl}{\begin{example}}
\newcommand{\eexmpl}{\end{example}}

\newcommand{\beqn}{\begin{equation}}
\newcommand{\eeqn}{\end{equation}}
\newcommand{\beqns}{\begin{equation*}}
\newcommand{\eeqns}{\end{equation*}}

\newcommand{\RRR}{\mathbb{R}^3}

\newcommand{\ol}{\overline}
\newcommand{\Stwo}{{\mathbb{S}^2}}

\renewcommand{\leq}{\leqslant}
\renewcommand{\geq}{\geqslant}

\definecolor{mygreen}{rgb}{0.1,0.75,0.2}

\newcommand{\N}{\mathbb{N}}

\newcommand{\eps}{\epsilon}

\newcommand{\E}{\mathsf{E}}
\newcommand{\fancE}{\mathscr{E}}

\newcommand{\Om}{\Omega}

\newcommand{\calU}{\mathcal{U}}
\newcommand{\calV}{\mathcal{V}}

\DeclareMathOperator{\Per}{Per}
\DeclareMathOperator{\dive}{div}

\DeclareMathOperator{\loc}{loc}
\title[Liquid drop and TFDW models with long-range attraction]{Ground-states for the liquid drop and TFDW models with long-range attraction}
\author{Stan Alama}
\address{Dept of Math and Stat, McMaster University, Hamilton, ON}
\email{alama@mcmaster.ca}
\author{Lia Bronsard}
\address{Dept of Math and Stat, McMaster University, Hamilton, ON}
\email{bronsard@mcmaster.ca}
\author{Rustum Choksi}
\address{Dept of Math and Stat, McGill University, Montr\'{e}al, QC}
\email{rustum.choksi@mcgill.ca}
\author{Ihsan Topaloglu}
\address{Dept of Math and Appl Math, Virginia Commonwealth University, Richmond, VA}
\email{iatopaloglu@vcu.edu}

\date{\today}                                        

\begin{document}

\begin{abstract}
We prove that both the liquid drop model in $\R^3$ with an attractive background nucleus and the Thomas-Fermi-Dirac-von Weizs\"{a}cker (TFDW) model attain their ground-states \emph{for all} masses as long as the external potential $V(x)$ in these models is of long range, that is, it decays slower than Newtonian (e.g., $V(x)\gg |x|^{-1}$ for large $|x|$.) For the TFDW model we adapt classical concentration-compactness arguments by Lions, whereas for the liquid drop model with background attraction we utilize a recent compactness result for sets of finite perimeter by Frank and Lieb.
\end{abstract}

\maketitle

%\baselineskip=18pt
%%%%%%%%%%%%%%%%%%%%%%%%%%%%%%%%%%%%%%%%%%%%%%%%%%%%%%%%%%%%%%%%%%%%%%%%%%%%%%%%%%%%%%%%%%%%%%%%%%%
%%%%% INTRODUCTION
%%%%%%%%%%%%%%%%%%%%%%%%%%%%%%%%%%%%%%%%%%%%%%%%%%%%%%%%%%%%%%%%%%%%%%%%%%%%%%%%%%%%%%%%%%%%%%%%%%%
\section{Introduction}\label{sec:intro}

In this note we consider ground-states of two mass-constrained variational problems containing an external attractive potential to the origin which is super-Newtonian at long ranges. The first problem 
consists of a variant of Gamow's liquid drop problem \cite{Ga1930, We1935, KnMu2014} perturbed by an attractive background  potential $V(x)$, with long range decay, in the sense that $V(x)\gg |x|^{-1}$ for large $|x|$. 
The second problem is a variant of the Thomas-Fermi-Dirac-von Weizs\"{a}cker (TFDW) functional, again subject to an external attractive potential $V(x)$ which is ``super-Newtonian''. 

Let us first state the two problems precisely. 
The variant of the liquid drop problem is given by
	\beqn \label{eqn:e_Z}
		e_V(M) \,:=\, \inf \left\{ \E_V(u) \colon  u\in BV(\RRR;\{0,1\}), \ \int_{\RRR} u\,dx=M\right\}, \tag{LD}
	\eeqn
where the energy functional $\E_V$ is defined as
	\beqn \label{eqn:energy}
		\E_V(u) := \int_{\RRR} |\nabla u| + \int_{\RRR}\!\int_{\RRR} \frac{u(x)u(y)}{|x-y|}\,dxdy - \int_{\RRR} V(x)u(x)\,dx.
	\eeqn
Here the first term in $\E_V$ computes the total variation of the function $u$, i.e.,
	\[
		\int_{\RRR} |\nabla u| = \sup\left\{ \int_{\RRR} u\dive\phi\,dx \colon \phi \in C_0^1(\RRR;\RRR), \ |\phi| \leq 1  \right\}
	\]
and is equal to $\Per_{\RRR}(\{x\in\RRR \colon u(x)=1\})$ since $u$ takes on only the values 0 and 1.

The variant of the TFDW problem we consider here is to find
	\beqn \label{eqn:I_Z}
		I_V(M) \,:=\, \inf \left\{ \fancE_V(u) \colon u\in H^1(\RRR), \ \int_{\RRR} |u|^2\,dx=M \right\}, \tag{TFDW}
	\eeqn
where
	\beqn \label{eqn:TFDW_energy}
		\fancE_V(u) := \int_{\RRR} \Bigg( |\nabla u|^2 + |u|^{10/3} - |u|^{8/3} - V(x)|u|^2 \Bigg)dx + \frac{1}{2} \int_{\RRR}\!\int_{\RRR} \frac{|u(x)|^2 |u(y)|^2}{|x-y|}\,dxdy.
	\eeqn

In the original TFDW problem (see the works of Benguria, Br\'{e}zis, Lieb \cite{BeBrLi1981}, Le Bris, Lions \cite{LeLi2005}, and Lieb \cite{Lieb1981} for detailed surveys on this classical theory), the potential is taken to be 
	\[ 
		V_Z(x) := \frac{Z}{|x|}, 
	\]
simulating an attracting  point charge at the origin with charge $Z$. 
With this physical choice of potential, both the liquid drop and TFDW  problems have been shown to exhibit existence for small $M$  and nonexistence for large $M$.  
In particular, for the liquid drop model it has recently been shown by  Lu and Otto, and by Frank, Nam and van den Bosch that
	\begin{itemize}
		\item (nonexistence, Theorem 1.4 of Frank, Nam, van den Bosch \cite{FrNaBo2016}) if $\E_{V_Z}$ has a minimizer, then $M \leq \min \{ 2Z+8, Z+CZ^{1/3}+8\}$ for some $C>0$; and,
		\item (existence, Theorem 2 of Lu, Otto \cite{LO2}) there exists a constant $c>0$ so that for $M \leq Z+c$ the unique minimizer of $\E_{V_Z}$ is given by the ball $\chi_{B(0,R)}$, 
	\end{itemize}
where $R=(M/\omega_3)^{1/3}$ and $\omega_3$ denotes the volume of the unit ball in $\RRR$.
Similar (and older) existence results hold for the TFDW problem.  The existence of solutions to the classical TFDW problem was established by Lions \cite{Lions1987} for $M\leq Z$  and extended to $M\leq Z+c$ for some constant $c>0$ by Le Bris \cite{LeBris1993}. The nonexistence of ground-states for large values of $M$ (or small values of $Z$) is only recently proved by Frank, Nam and van den Bosch \cite{FrNaBo2016,NaBo17}. In a separate paper Nam and van den Bosch \cite{NaBo17} also consider more general external potentials which are short-ranged, i.e., $\lim_{|x|\to\infty} |x|V(x)=0$.  Motivated by their result, here we look at the complementary case, in which the external potential is asymptotically {\em larger} than Newtonian at infinity.

These functionals can be viewed as mathematical paradigms for the existence and nonexistence of coherent structures based upon a mass parameter. Since both problems are driven by a repulsive potential of Coulombic (Newtonian) type, it is natural to expect that if the confining external potential $V$ was even slightly stronger (at 
long ranges) than Newtonian, global existence would be restored for all masses. In this note we prove that this is indeed the case.

For the liquid drop problem $e_V$, we consider the external potentials $V$ which satisfy the following hypotheses:
\medskip
		\begin{enumerate}[label=(H{\arabic*})]
				\addtolength{\itemsep}{8pt}
						\item \label{itm:H1} $V\geq 0$, and $V\in L^1_{\loc}(\RRR)$.
						%, and for every $m>0$ there exists a constant $C=C(m)>0$ such that $\disp \int_K V\,dx \leq C$ for all compact $K\subset\RRR$ with $|K|=m$.
						%\item \label{itm:H2} $V \in C^1(\RRR\setminus\{0\})$, and $\nabla V(x) \cdot x \leq 0$ for all $x\in\RRR\setminus\{0\}$.
						%\item \label{itm:H3} There exists $\alpha>0$ such that $\int_{B_\eps(0)} |V(x)| \,dx \leq C\eps^{\alpha}$ for all $\eps>0$.
						\item \label{itm:H4} $\disp \lim_{t\to\infty} t \,\left( \inf_{|x| = t} V(x)\right)= \infty$.
						\item \label{itm:H5} $\disp \lim_{|x|\to\infty} V(x) = 0$.
		\end{enumerate}	
\medskip

On the other hand, to ensure that the energy $\fancE_V$ is bounded below, we assume that $V$ satisfies
\medskip
		\begin{enumerate}[label=(H{\arabic*}$^\prime$)]
				\addtolength{\itemsep}{8pt}
						\item \label{itm:H1pr} $V\geq 0$, and $V \in L^{3/2}(\RRR)+L^\infty(\RRR)$,
		\end{enumerate}	
\medskip
instead of \ref{itm:H1}, along with \ref{itm:H4} and \ref{itm:H5}. 
\medskip
Hypothesis \ref{itm:H4} implies that these potentials are long-ranged but only slightly more attractive than Newtonian. A typical example of such an external potential is
	\[
	 V(x) = \frac{1}{|x|^{1-\eps}}
	\]
for $0<\eps<1$, or a linear combination of functions of this form. 
Although these potentials have only slightly longer range than $|x|^{-1}$, this is sufficient to ensure existence of ground-states for the modified liquid drop and TFDW problems, $e_V$ and $I_V$, \emph{for all} $M>0$.

\bthm[Liquid drop model] \label{thm:existence1}
 Suppose $V$ satisfies {\rm\ref{itm:H1}--\ref{itm:H5}}. Then for any $M>0$ the problem $e_V(M)$ given by \eqref{eqn:e_Z} has a solution.
\ethm

\medskip

\bthm[TFDW model] \label{thm:existence2}
Suppose $V$ satisfies {\rm\ref{itm:H1pr}, \ref{itm:H4}}, and {\rm\ref{itm:H5}}. Then for any $M>0$ the problem $I_V(M)$ given by \eqref{eqn:I_Z} has a solution.
\ethm

\medskip

\begin{remark}\label{rem:splitting}
While we do obtain existence of ground-states for all masses $M$, we do {\em not} expect that the attractive potential $V$ stabilizes the single droplet solution $\chi_{B(0,(M/\omega_3))^{1/3}}$ for large values of $M$.  Rather, we expect that mass splitting does indeed occur (as it does for the unperturbed liquid drop problem \cite{LO1,KnMu2014,KnMuNo2016}) but the resulting components are confined by the external potential $V$ and cannot escape to infinity. This expectation is reflected in our approach to the proof of the two theorems above.
\end{remark}

\medskip

While the mathematical motivations for these results are clear, let us now comment on 
the physicality of the long-range super-Newtonian attraction.  
For the quantum TFDW model, we do not know of any physical situation which would support  an ``exterior'' potential  producing  super-Newtonian attraction. 
However we note that these functionals, in particular the liquid drop energy, can be used as phenomenological models for charged or gravitating masses at all length scales.  
Consideration of super-Newtonian forces appears in several theories at the cosmological level, and in fact the validity of Newton's law at long distances has been a longstanding interest in physics. As Finzi notes \cite{Finzi1963}, for example, stability of cluster of galaxies implies stronger attractive forces at long distances than that predicted by Newton's law. Motivated by similar observations, Milgrom \cite{Milgrom1983} introduced the modified Newtonian dynamics (MOND) theory which suggests that the gravitational force experienced by a star in the outer regions of a galaxy must be stronger than Newton's law (see also works of Bugg \cite{Bugg2015}, and Milgrom \cite{Milgrom2015} for a survey, and Bekenstein's work \cite{Bekenstein2004}).

\medskip

\subsection*{Outline of the paper:}  The proofs of Theorems \ref{thm:existence1} and \ref{thm:existence2} follow the same basic strategy:  to obtain a contradiction, we assume that minimizing sequences lose compactness, and use concentration compactness techniques to show that it is because of the splitting and dispersion of mass to infinity (``dichotomy'').  For the liquid drop model, we utilize a recent technical concentration-compactness result for sets of finite perimeter by Frank and Lieb \cite{FrLi2015} to prove a lower bound on the energy in case minimizing sequences $u_n$ lose compactness via splitting, of the form
\beqn\label{eq:lower}
 \lim_{n\to\infty}\E_V(u_n) \geq e_V(m_0) + e_0(m_1 )+ e_0(M-m_0-m_1 ),  
\eeqn
where $0<m_i<M$ with $m_0+m_1\leq M$. However, thanks to the super-Newtonian decay of $V$ we then show that $e_V(M)$ actually lies strictly below the value given in \eqref{eq:lower}.  This is a variant on the original ``strict subadditivity'' argument introduced by Lions \cite{Lions1987} for the classical TFDW model with $V(x)=|x|^{-1}$, and subsequently used in innumerable treatments of variational problems with loss of compactness.

In section 3 we adapt recent arguments by Nam and van den Bosch \cite{NaBo17} along with estimates of Lions \cite{Lions1987} and Le Bris \cite{LeBris1993} to prove Theorem \ref{thm:existence2}.  Although the variational structure of TFDW is nearly the same as the liquid drop model, the components are not compactly supported, so we require an additional argument to verify that they decay sufficiently rapidly (in fact exponentially) in order to calculate the interaction between components.

%%%%%%%%%%%%%%%%%%%%%%%%%%%%%%%%%%%%%%%%%%%%%%%%%%%%%%%%%%%%%%%%%%%%%%%%%%%%%%%%%%%%%%%%%%%%%%%%%%%
%%%%%%%%%%%%%%%%%%%%%%%%%%%%%%%%%%%%%%%%%%%%%%%%%%%%%%%%%%%%%%%%%%%%%%%%%%%%%%%%%%%%%%%%%%%
\section{Proof of Theorem \ref{thm:existence1}} \label{sec:thm_i}

Our proof relies on a recent concentration-compactness type result for sets of finite perimeter by Frank and Lieb \cite{FrLi2015}. 
While similar compactness results are known and could be adapted here (for example, the classical theory of Lions  \cite{Lions84}, and 
results for minimizing clusters which can be found in Chapter 29 of Maggi \cite{Maggi}), the results of Frank and Lieb are particularly well-suited for our purposes. Throughout the proof of Theorem \ref{thm:existence1}, we specifically use Proposition 2.1, and Lemmas 2.2 and 2.3 of Frank and Lieb \cite{FrLi2015}.

As noted in the introduction our goal is to obtain a splitting property \eqref{eq:lower} for $e_V(M)$ involving the ``minimization problem at infinity'' $e_0$ given by
	\beqn \nonumber
		e_0 (M)\, := \, \inf \left\{ \E_0(u) \colon u\in BV(\RRR;\{0,1\}), \ \text{and} \ \int_{\RRR} u\,dx=M  \right\},
	\eeqn
where
	\beqn	\nonumber
	  \E_0(u)\,	:=\, \int_{\RRR} |\nabla u| + \int_{\RRR}\!\int_{\RRR} \frac{u(x)u(y)}{|x-y|}\,dxdy.
	\eeqn
%The property \eqref{eqn:sub_add} along with a matching lower bound obtained via a compactness argument will show that any minimizing sequence localizes and is relatively compact.

\bigskip

We will also use the following simple weak compactness result for the confinement term, which is convenient to state in general terms. 

\begin{lemma}\label{lem:confinement}
Let $A_n\subset\RRR$ be a sequence of sets with $|A_n|\leq M$ which converge to zero locally, i.e., $\chi_{A_n}\to 0$ in $L^1_{\loc}(\RRR)$. Then
	\[
	\int_{A_n}V\,dx=\int_{\RRR} V \chi_{A_n}\,dx \to 0 \qquad \text{as} \quad n\to\infty.
	\]
\end{lemma}

\begin{proof}
By hypothesis \ref{itm:H5}, for any $\epsilon>0$, there exists $R>0$ so that if $V_\infty:=V\chi_{B^c_R}$, then $0 \leq V_\infty <{\epsilon\over 3M}$. By \ref{itm:H1}, on the other hand, we define $V_1:=V\chi_{B_R\setminus E_K}$, where $E_K=\{x\in B_R \colon  0\leq V(x)\leq K\}$ and $K$ is chosen with $\| V_1\|_{L^1(\RRR)}<{\epsilon\over 3}$. Finally, let $V_2:=V\chi_{E_K}$, which is supported in $B_R$, and satisfies $\|V_2\|_{L^\infty(\RRR)}\leq K$.
%
%By \ref{itm:H1} and \ref{itm:H5}, we may decompose $V=V_1+V_2+V_\infty$, as follows:
%\begin{itemize} \addtolength{\itemsep}{6pt}
%\item[(i)] For any $\epsilon>0$, there exists $R>0$ so that if $V_\infty=V\chi_{B^c_R}$, then $0 \leq V_\infty <{\epsilon\over 3M}$.
%\item[(ii)]  $V_1=V\chi_{B_R\setminus E_K}$, where $E_K=\{x\in B_R \colon  0\leq V(x)\leq K\}$ and $K$ is chosen with $\| V_1\|_{L^1(\RRR)}<{\epsilon\over 3}$.
%\item[(iii)]  $V_2=V\chi_{E_K}$ is supported in $B_R$ with $\|V_2\|_{L^\infty(\RRR)}\leq K$.
%\end{itemize}

Now with these choices we have a decomposition of $V$ into $V_1+V_2+V_\infty$, depending on $\epsilon$ and $K$. Using this decomposition
	\[
	  0\leq \int_{A_n} V \,dx  \leq  \|V_1\|_{L^1} + K|A_n\cap B_R| +  {\epsilon\over 3M}|A_n|
   <  K|A_n\cap B_R| + {2\epsilon\over 3} <\epsilon,
	\] 
for all $n$ large enough, since $|A_n\cap B_R|\to 0$ as $n\to\infty$ by local convergence of the sets $A_n$.
\end{proof}

\bigskip

\begin{proof}[Proof of Theorem \ref{thm:existence1}] 
First, by \ref{itm:H1} and \ref{itm:H5} we may write $V=V\chi_{B_R}+ V\chi_{B_R^c}\in L^1+L^\infty$, where $R$ is chosen so that $\|V\chi_{B_R^c}\|_{L^\infty(\RRR)}\leq 1$.  Then, for any $u=\chi_\Om$ with $|\Om|=M$,
	\[
		  \int_{\RRR} V u \,dx \leq \|V\|_{L^1(B_R)} + M,
	\]
hence, $e_V(M)>-\infty$.  Now, let $\{u_n\}_{n\in\mathbb{N}} \subset BV(\RRR;\{0,1\})$ with $\int_{\RRR}u_n\,dx=M$ be a minimizing sequence for the energy $\E_V$, i.e., $\lim_{n\to\infty} \E_V(u_n) = e_V(M)$. By the above estimate on the confinement term, the minimizing sequence has uniformly bounded perimeter, $\int_{\RRR} |\nabla u_n|\leq C$ independent of $n$. Define the sets of finite perimeter $\Om_n \subset \RRR$ so that $\chi_{\Om_n} = u_n$, and $|\Om_n|=M$ for all $n\in\N$.

\smallskip

\noindent \emph{Step 1.} First, we set up our contradiction argument.
 By the compact embedding of $BV(\RRR)$ in $L^1_{\loc}(\RRR)$ (see e.g. Corollary 12.27 in Maggi \cite{Maggi}) there exists a subsequence and a set of finite perimeter $\Om^0\subset\R^3$ so that $\Om_n\to\Om^0$ locally, that is, $u_n\to \chi_{\Om^0}:=w^0$ in $L^1_{\loc}(\RRR)$. At this point, we admit the possibility that $w^0\equiv 0$, i.e., $|\Om^0|=0$. However, in Step 4 we show that $w^0\not\equiv 0$.
 
If the limit set $|\Om^0|=M$, then we are done. Indeed, since $\{u_n\}_{n\in\N}$ is locally convergent in $L^1$, a subsequence converges almost everywhere in $\RRR$.  In addition, the norms converge, $\|u_n\|_{L^1}=M=\|\chi_{\Om^0}\|_{L^1}$, so by the Brezis-Lieb Lemma (see Theorem 1.9 in Lieb and Loss \cite{LiebLoss97}) we may then conclude that (along a subsequence) $u_n\to w^0=\chi_{\Om^0}$ in $L^1$ norm.  By the lower semicontinuity of the perimeter (Proposition 4.29 in Maggi \cite{Maggi}) and of the interaction terms (Lemma 2.3 of Frank and Lieb \cite{FrLi2015})
	\[
	  \int_{\RRR} |\nabla w^0 | \leq \liminf_{n\to\infty}\int_{\RRR} |\nabla u_n | \qquad
   \int_{\RRR}\!\int_{\RRR} \frac{w^0(x)w^0(y)}{|x-y|}\,dxdy \leq \liminf_{n\to\infty}
      \int_{\RRR}\!\int_{\RRR} \frac{u_n(x)u_n(y)}{|x-y|}\,dxdy.
	\]    
To pass to the limit in the confinement term, we apply Lemma~\ref{lem:confinement} to the sequence $u_n-w^0\to 0$ in $L^1(\RRR)$, and together with the above we have
	\[
		\E_V(w^0) \leq \liminf_{n\to\infty} \E_V(u_n).
	\]
Therefore we conclude that $w^0=\chi_{\Omega^0}$ attains the minimum value of $E_V$, and the proof is complete.  To derive a contradiction, we now assume that $m_0:=|\Omega^0|<M$.  

\medskip
 
\noindent \emph{Step 2.}  Next, we show that the energy splits.  
First, assume that $0 <|\Om^0|<M$. We apply Lemma 2.2 of Frank and Lieb \cite{FrLi2015} (with $x_n=x^0_n=0$):  there exists $r_n>0$ such that the sets
	\[
		\calU_n^0 = {\Om}_n \cap B_{r_n} \quad\text{ and } \quad \calV_n^0 = {\Om}_n \cap (\RRR \setminus \ol{B}_{r_n})
	\]
satisfy
\begin{gather*}
		\chi_{\calU_n^0} \to \chi_{\Om^0} \quad \text{in } L^1(\RRR), \quad \chi_{\calV_n^0}\to 0 \quad \text{in } L^1_{\loc}(\RRR),  \\  
		\lim_{n\to\infty} |\calU_n^0| = |\Om^0|=m_0, \qquad
		\Per \Om^0\leq \liminf_{n\to\infty} \Per \calU_n^0,
	\\
		\text{and}\quad 
		\lim_{n\to\infty} (\Per \Om_n -\Per \calU_n^0 -\Per \calV_n^0)=0.
\end{gather*}

We now define $w_n^0(x) := \chi_{\calU^0_n}(x)$, $w^0(x) := \chi_{\Om^0}(x)$, $\Om^0_n:=\calV^0_n$, and $u_n^0(x) := \chi_{\Om^0_n}(x)$ so that $u_n = w_n^0 + u_n^0=w^0+ u_n^0 +o(1)$ in $L^1(\RRR)$, and $u_n^0\to 0$ in $L^1_{\loc}$.  
In particular, by Lemma~\ref{lem:confinement}, 
$$  \int_{\RRR} V u_n \,dx = \int_{\RRR} V w^0 \,dx + o(1).  $$
Using Lemma 2.3 in Frank and Lieb \cite{FrLi2015}, the nonlocal interaction term in $\E_V$ splits in a similar way as the perimeter, 
\begin{align*}
  \int_{\RRR}\int_{\RRR} {u_n(x)\, u_n(y)\over |x-y|}\, dx dy
  &= \int_{\RRR}\int_{\RRR}{w_n^0(x)\, w_n^0(y)\over |x-y|}\, dx dy + 
   \int_{\RRR}\int_{\RRR} {u_n^0(x)\, u_n^0(y)\over |x-y|}\, dx dy + o(1)\\
   &= \int_{\RRR}\int_{\RRR}{w^0(x)\, w^0(y)\over |x-y|}\, dx dy + 
   \int_{\RRR}\int_{\RRR} {u_n^0(x)\, u_n^0(y)\over |x-y|}\, dx dy + o(1),
\end{align*}
and thus the energy splits, up to a small error,
	\beqn \label{eq:splits}
		\E_V(u_n) = \E_V(w_n^0) + \E_0(u_n^0) + o(1)\geq \E_V(w^0) + \E_0(u_n^0) + o(1).  
	\eeqn

In the case $|\Om^0|=0$ (which we eliminate in Step 4 below,) this splitting becomes trivial, with $w^0\equiv 0$ and $u_n^0=u_n$.

\medskip

\noindent \emph{Step 3.} Now we repeat the above procedure to locate a concentration set for the remainder $u_n^0$.  We argue as above, but with $u_n^0$ replacing $u_n$, that is, the remainder set $\Om^0_n=\calV^0_n$ replacing $\Om_n$. We know that $u_n^0=\chi_{\Om_n^0}\to 0$ locally in $L^1(\RRR)$, $|\Om^0_n|=M-m_0+o(1)\in (0,M]$, and $\E_V(u_n^0)$ (and hence $\Per \Om_n^0$) are uniformly bounded.  By Proposition 2.1 in Frank and Lieb \cite{FrLi2015} there exists a set $\Om^1$ with $0<|\Om^1|\leq M-m_0$ and a sequence of translations $x_n\in\RRR$ such that for some subsequence
$\chi_{\Om^0_n- x_n}\to \chi_{\Om^1}$ in $L^1_{\loc}(\RRR)$.  Since $\chi_{\Om_n^0}\to 0$ $L^1_{\loc}(\RRR)$, we have that the translation points $|x_n|\to\infty$ as $n\to\infty$. 
Again, by Lemmas 2.2 and 2.3 of Frank and Lieb \cite{FrLi2015}, and Lemma~\ref{lem:confinement} as in Step 2, we similarly obtain a disjoint decomposition $\Om^0_n-x_n=\calU^1_n \cup \calV^1_n$, with $\chi_{\calU^1_n}\to \chi_{\Om^1}$ in $L^1(\RRR)$, $\chi_{\calV^1_n}\to 0$ in $L^1_{\loc}(\RRR)$, and for which the energy splits as in \eqref{eq:splits}, namely,
	\[
		  \E_V(u_n^0) =\E_0(u_n^0)+o(1)\geq \E_0(w^1)+ \E_0(u^1_n) + o(1),
	\]
where $w^1:=\chi_{\Om^1}$, $u^1_n=\chi_{\calV^1_n+x_n}\to 0$ in $L^1_{\loc}(\RRR)$, and $|\calV^1_n|=|\calV^0_n|-m_1+o(1)$.  We denote the re-centered remainder set $\Om^1_n:=\calV^1_n+x_n$, so that $u^1_n(x)=\chi_{\Om^1_n}(x)$.
Combining with the previous step, we now have
	\[
	  \E_V(u_n)\geq \E_V(w^0) + \E_0(w^1) + \E_0(u^1_n)+ o(1) \quad\text{and}\quad M=m_0+m_1+|\Om^1_n|+ o(1).
	\]
This, combined with the continuity of $e_0$ (see e.g. Lemma 4.8 in the work of Kn\"{u}pfer, Muratov and Novaga \cite{KnMuNo2016}) yields a lower bound estimate in case of splitting,
\beqn\label{eq:energysplit}
   e_V(M)\geq e_V(m_0) + e_0(m_1)+ e_0(M-m_0-m_1).
\eeqn

\medskip

\noindent \emph{Step 4.} We claim that $w^0 \not\equiv 0$. For a contradiction, assume $w^0\equiv 0$. Define a sequence by $w_n(x):=u_n(x+x_n)$ using the translation sequence found above. Then $w_n \to w^1$ in $L^1(\RRR)$ and $w_n-u_n^0 \to 0$ in $L^1_{\loc}(\RRR)$. This implies, by Lemma \ref{lem:confinement}, that $\lim_{n\to\infty} \int_{\RRR} V(w_n-u_n^0)\,dx=0$. Now this limit and the translation invariance of the first two terms of $\E_V$ yield
	\[
		\E_V(w_n)-\E_V(u_n) = -\int_{\RRR} V (w_n-u_n)\,dx \longrightarrow - \int_{\RRR} V w^1 dx  <0, 
	\]
hence, a contradiction.

\medskip

\noindent \emph{Step 5.} 
 Now we prove that $e_V(m_0)=\E_V(w^0)$  and $e_0(m_1)=\E_0(w^1)$. By subadditivity (see Lemma 4 of Lu and Otto \cite{LO2} and Lemma 3 in their earlier work \cite{LO1}, or Step 5 below) we have a rough upper bound estimate of the form
	$$
		e_V(M) \leq e_V(m_0) + e_0(m_1)+ e_0(M-m_0-m_1).
 	$$
 Combined with \eqref{eq:energysplit}, this yields
 	\begin{align*}
 			e_V(m_0) + e_0(m_1)+ e_0(M-m_0-m_1) &\geq e_V(M)  \\
			&\geq  \E_V(w^0) + \E_0(w^1)+\liminf_{n\to\infty} \E_0(u_n^1) \\
			  &\geq e_V(m_0) + e_0(m_1)+ e_0(M-m_0-m_1).
 	\end{align*}
 Hence,
 	\[
 		\big(\E_V(w^0)-e_V(m_0) \big) 
		+\big(\E_V(w^1)-e_V(m_1) \big) 
		 +  \big( \liminf_{n\to\infty} \E_0(u_n^1)-e_0(M-m_0-m_1)\big) = 0, 
 	\]
 and since every term in this sum are nonnegative we must conclude that
 	\[
 		\E_V(w^0)=e_V(m_0) \quad \text{ and }\quad \E_0(w^1)=e_0(m_1) . 	\]
 
 \medskip
 
\noindent \emph{Step 6.} Finally, we show, by means of an improved upper bound, that splitting leads to a contradiction, and hence the minimum must be attained.  It is here that we use the super-Newtonian  attraction hypothesis \ref{itm:H4}.  Since $w^0=\chi_{\Om^0}(x),\, w^1=\chi_{\Om^1}(x)$ are minimizers of $e_V$ and $e_0$ respectively, by regularity of minimizers\cite{KnMu2014,LO2} we may choose $R>0$ for which $\Om^0,\, \Om^1\subset B_R(0)$. Let $b\in\Stwo$ be any unit vector.  For $t$ sufficiently large so that $\Om^0\cap(\Om^1+tb)=\emptyset$, let  
$$  F(t):= \int_{\RRR}\!\int_{\RRR} \frac{w^0(x)w^1(y-tb)}{4\pi|x-y|}\,dxdy, \ \text{and} 
   \ G(t):=\int_{\RRR} V(x)w^1(x-tb)\,dx.  $$
We now estimate each; first,
\begin{align*}
F(t)
&\leq \int_{B_R(0)} \int_{B_R(tb)} {1\over 4\pi |x-y|}dx\, dy 
\leq {|B_R|^2\over 4\pi (t - 2R)} \leq {|B_R|^2\over 2\pi t},
\end{align*}
for all $t$ large enough. 

To estimate $G(t)$ from below, we recall from \ref{itm:H4} that for any $A>0$ there exists $t_1>0$ such that for all $t>t_1$, 
$$  \inf_{|x|=t} V(x) \geq {A\over t}.  $$
Thus, for each $i=1,\dots,N$, as $t\to\infty$,
$$  t\int_{\RRR} V(x) w^1(x-tb)\, dx = \int_{\Om^1} t V(x+tb)\, dx
     \geq \int_{\Om^1} {tA\over |x+tb|} dx\longrightarrow A|\Om^1|,
$$
by dominated convergence, and hence $\lim_{t\to\infty} tG(t)=\infty$.  Thus, $t(F(t)-G(t))\to -\infty$ as $t\to\infty$.  Choose $\eps>0$ and $t_0>0$ such that 
$$  F(t_0)-G(t_0)<-\eps<0.  $$

With this choice of $\eps>0$, we may choose a compact set $K=K(\eps)$ for which $|K|=M-m_0-m_1$ and
$$  \E_0(\chi_K)< e_0(M-m_0-m_1) + {\eps\over 3}.  $$
Choose $\tau>0$ large enough so that $K_\tau:=K-\tau b$ satifies
%$$  \int_{K_\tau} V(x)\, dx, \ \int_{\Om^i} \int_{K_\tau} {1\over 4\pi |x-y|} dx\, dy < {\eps\over 5}, \qquad i=0,1.  $$
 $$ \int_{\Om^i} \int_{K_\tau} {1\over 4\pi |x-y|} dx\, dy < {\eps\over 3}, \qquad \text{for}\quad i=0,1. $$
Using $v(x)=w^0(x)+w^1(x-t_0 b) + \chi_{K_\tau}$ as a test function, which is admissible for 
$e_V(M)$, we have
\begin{align*}
e_V(M)\leq \E_V(v) &= \E_V(w^0) +\E_0(w^1) + \E_0(\chi_{K_\tau}) + F(t_0)-G(t_0) \\
	&\qquad +
                   \sum_{i=0,1} \int_{\Om^i} \int_{K_\tau} {1\over 4\pi |x-y|}\, dxdy
                                   - \int_{K_\tau} V(x)\, dx \\
	&\leq e_V(m_0)+e_0(m_1) + e_0(M-m_0-m_1) -{\eps\over 3},
\end{align*} 
which contradicts the lower bound in case of splitting, \eqref{eq:energysplit}.  Thus we must have $|\Om^0|=M$ and $e_V(M)=\E_V(w^0)$, for any $M>0$.
\end{proof}
 
%\edred{\begin{remark}\label{rem:referee}
%As pointed out by the referee, there is an alternative method to arrive at a contradiction with only one iteration of the concentration-compactness argument, and we employ this strategy in the second part; however the simplification for the proof of Theorem \ref{thm:existence1} is minor.
%\end{remark}
%}

%%%%%%%%%%%%%%%%%%%%%%%%%%%%%%%%%%%%%%%%%%%%%%%%%%%%%%%%%%%%%%%%%%%%%%%%%%%%%%%%%%%%%%%%%%%%%%%%%%%
%%%%%%%%%%%%%%%%%%%%%%%%%%%%%%%%%%%%%%%%%%%%%%%%%%%%%%%%%%%%%%%%%%%%%%%%%%%%%%%%%%%%%%%%%%%%%%%%%%%
\section{Proof of Theorem \ref{thm:existence2}} \label{sec:thm_ii}
%%%%%%%%%%%%%%%%%%%%%%%%%%%%%%%%%%%%%%%%%%%%%%%%%%%%%%%%%%%%%%%%%%%%%%%%%%%%%%%%%%%%%%%%%%%%%%%%%%%
%%%%%%%%%%%%%%%%%%%%%%%%%%%%%%%%%%%%%%%%%%%%%%%%%%%%%%%%%%%%%%%%%%%%%%%%%%%%%%%%%%%%%%%%%%%%%%%%%%%

Now we turn our attention to $\fancE_V$ and $I_V(M)$ given by \eqref{eqn:TFDW_energy} and \eqref{eqn:I_Z}, respectively. As in the previous section, we define the ``problem at infinity'' by
	\beqn \nonumber
		I_0(M) \,:=\, \inf \left\{ \fancE_0(u) \colon u\in H^1(\RRR), \ \int_{\RRR} |u|^2\,dx=M \right\},
	\eeqn
where
	\beqn \nonumber
		\fancE_0(u) := \int_{\RRR} \Bigg( |\nabla u|^2 + |u|^{10/3} - |u|^{8/3} \Bigg)dx + \frac{1}{2} \int_{\RRR}\!\int_{\RRR} \frac{|u(x)|^2 |u(y)|^2}{|x-y|}\,dxdy.
	\eeqn

First we note that the problems $I_V$ and $I_0$ satisfy the following ``binding inequality'', which is the standard subadditivity condition from concentration-compactness principle. For the proof of the following lemma we refer to Lemma 5 in Nam and van den Bosch \cite{NaBo17}.

\begin{lemma} \label{lem:binding}
For all $0<m<M$ we have that  
	\[
		I_V(M)\leq I_V(m)+ I_0(M-m).
	\]
Moreover, $I_V(M)<I_0(M)<0$, $I_V(M)$ is continuous and strictly decreasing in $M$.
\end{lemma}

Next we prove that the ground-state value $I_V(M)$ is bounded.

\begin{lemma}\label{lem:3lem1} Let $\{u_n\}_{n\in\N}\subset H^1(\RRR)$ be a minimizing sequence for the energy $\fancE_V$ with $\int_{\RRR} |u_n|^2\,dx=M$. Then there exists constant $C_0>0$ such that $\|u_n\|_{H^1(\RRR)}^2 \leq C_0\,M$.
\end{lemma}

\begin{proof}
First, we note that $I_V(M)<0$ for any $M>0$.  Indeed, in the proof Lemma 5 of Nam and van den Bosch \cite{NaBo17} it is shown that $I_0(M)<0$, and $\fancE_V(u)\leq \fancE_0(u)$ holds for all $u\in H^1(\RRR)$ with $\int_{\RRR} |u|^2\,dx=M$. We first claim that the quadratic form defined by the Schr\"odinger operator $-\Delta - V(x)$ is bounded below, i.e., that there exists $\lambda>0$ with
	\[
	   \int_{\RRR} \big( |\nabla u|^2 - V(x)|u|^2\big) dx \geq \frac12 \|u\|^2_{H^1}-\lambda\|u\|_{L^2}^2,
	\]
for all $u\in H^1(\R^3)$.  To see this, we note that by \ref{itm:H1pr} we may write $V=V_1+V_2$ with $V_1\in L^{3/2}(\RRR)$ and $V_2\in L^\infty(\RRR)$. Moreover, we may assume that $\| V_1\|_{L^{3/2}(\RRR)}<\eps$ for some $\eps>0$ to be chosen later. By the H\"older and Sobolev inequalities it follows that
	\[
		  \int_{\RRR} |V_1|\, |u|^2\, dx \leq \| V_1\|_{L^{3/2}(\RRR)} \|u\|_{L^6(\RRR)}^2 \leq \eps\, S_3 \, \|\nabla u\|_{L^2(\RRR)}^2,
	\]
where $S_3>0$ is the Sobolev constant.  Thus,
	\[
		  \int_{\RRR} \left( |\nabla u|^2 - V(x) |u|^2 \right) dx \geq (1-\eps\, S_3) \|\nabla u\|_{L^2(\RRR)}^2 - \|V_2\|_{L^\infty(\RRR)}\|u\|_{L^2(\RRR)}^2,
	\]
and the lower bound is obtained by choosing
	\[
		\eps=\frac{1}{2S_3} \quad \text{ and } \quad \lambda=\|V_2\|_{L^\infty(\RRR)}+\frac{1}{2}.
	\]

Using the elementary inequality
	\[  
		 |u|^{10/3} - |u|^{8/3} = \left( |u|^{5/3} -\frac12 |u|\right)^2 - \frac14 |u|^2 \geq -\frac14 |u|^2
	\]
to estimate the nonlinear potential, we obtain the lower bound
\begin{align*}
 \fancE_V(u_n) &\geq \int_{\RRR} \big(  |\nabla u_n|^2 - V(x)|u_n|^2 \big)\, dx - \frac{1}{4} \|u_n\|_{L^2(\RRR)}^2 \\
&\geq  \frac12 \|u_n\|_{H^1}^2
  -\left(\lambda+\frac{1}{4}\right)\|u_n\|_{L^2(\RRR)}^2 
  = \frac12 \|u_n\|_{H^1}^2- {C_0\over 2}\,M
\end{align*}
 Since $I_V(M)<0$, for $n\in\N$ sufficiently large we have that $\fancE_V(u_n) <0$. Referring back to the above inequalities we obtain $\|u_n\|_{H^1(\RRR)}^2 \leq C_0\,M$.
\end{proof}

We now begin the proof of Theorem \ref{thm:existence2}. 
% \edred{We remark that we could follow the same strategy for Theorem~\ref{thm:existence2} as we did for the proof of Theorem~\ref{thm:existence1}, but instead present a slightly different approach, exploiting a very elegant and general concentration-compactness lemma due to Nam and van den Bosch \cite{NaBo17}.}

\begin{proof}[Proof of Theorem \ref{thm:existence2}]
Let $\{u_n\}_{n\in\N}$ be a minimizing sequence for the energy functional $\fancE_V$ such that $\int_{\RRR} |u_n|^2\,dx=M$. 

\smallskip

\noindent \emph{Step 1.} First, note that by the uniform $H^1$-bound in Lemma \ref{lem:3lem1} we may extract a subsequence so that $u_n \rightharpoonup v^0$ weakly in $H^1(\RRR)$ and strongly in $L^q_{\loc}(\RRR)$ for all $2\leq q<6$.  Let $v_n:= u_n- v^0$, so $v_n\rightharpoonup 0$ weakly in $H^1(\RRR)$ and strongly in $L^q(\RRR)$ on compact sets as $n\to\infty$.  In particular, by hypotheses \ref{itm:H1pr}, \ref{itm:H5} we have that
\begin{equation}\label{vanish}
\int_{\RRR} V(x) |v_n|^2\, dx \to 0
\end{equation}
as $n\to \infty$.
Combining this with the arguments in equations (62)--(64) of Nam and van den Bosch \cite{NaBo17} we may conclude that the energy $\fancE_V$ splits as
\begin{equation} \label{split1}
\lim_{n\to\infty} \left( \fancE_V(u_n) - \fancE_V(v^0) - \fancE_0(v_n) \right) =0.
\end{equation}
(Note that at this point it is possible that $v^0=0$, i.e., the first component is trivial, but later we will in fact show that $v^0 \not\equiv 0$, and thus it is a ground-state of $\fancE_V$.)
%
%We now claim that $v^0\neq 0$.  Indeed, if it were trivial then $\int |v_n|^2 = M$ and from \eqref{split1} we would conclude that 
%$$ I_0(M)>I_V(M)=\lim_{n\to\infty} \fancE_V(u_n)=\lim_{n\to\infty} \fancE_0(v_n)\geq I_0(M),  $$
%a contradiction.
Define
	\[
	  m_0 \,:=\, \int_{\RRR} |v^0|^2 \, dx \in [0,M].
	\]
Note also that weak convergence implies $\|v_n\|^2_{L^2(\RRR)}\to M-m_0$.
In case $m_0>0$, we observe that \eqref{split1} also implies
	\[
	  I_V(M) = \fancE_V(v^0) + \lim_{n\to\infty} \fancE_0(v_n) 
	  \geq I_V(m_0)+ \lim_{n\to\infty}I_0(\|v_n\|^2_{L^2(\RRR)}) 
	  = I_V(m_0)+ I_0(M-m_0),
	\]
by the continuity of $I_0$.
As the result of Lemma~\ref{lem:binding} gives the opposite inequality, we conclude that
	\[
	  I_V(M)= I_V(m_0) + I_0(M-m_0).
	\]  
In addition, $\fancE_V(v^0)=I_V(m_0)$; hence, $v^0$ is a ground-state, and $\{v_n\}_{n\in\N}$ is a minimizing sequence for $I_0(m_1)$ with $m_1=M-m_0$, i.e., $I_0(m_1)=\lim_{n\to\infty}\fancE_0(v_n)$.

\smallskip

If $m_0=M$ then the minimizing sequence is compact, and the proof is complete.
Therefore, we will assume for the remainder of the proof that $m_0<M$.

\medskip

\noindent \emph{Step 2.}   Concentration-compactness for $0\leq m_0<M$: there is a subsequence of $\{u_n\}$ (not relabeled), a sequence of points $\{y_n\}\subset \RRR$, constants $m_i>0$, and functions $v^i\in H^1(\R^3)$ for $i=0,1$ with 
\beqn \label{energydecomp} \left.
	\begin{gathered}
u_n - \big(v^0 + v^1(\,\cdot\, - y_n)\big) \to 0 \quad \text{ in } L^2(\R^3),  \\
m_0+m_1 \leq M, \quad \int_{\RRR} |v^i|^2\,dx = m_i, \quad
\fancE_V(v^0) = I_V(m_0), \quad \fancE_0(v^1)=I_0(m_1),   \\
\text{and } \ I_V(M) = I_V(m_0) + I_0(m_1) + I_0(M-m_0-m_1). 
	\end{gathered}\right\}
\eeqn

This concentration-compactness result is very similar to Steps 1--3 of the proof of Theorem~\ref{thm:existence1} , and in fact it follows immediately from steps (i) and (ii) of the proof of Lemma 9 of Nam and van den Bosch \cite{NaBo17}. (See also the Appendix of Lions \cite{Lions1987}.)

\medskip

\noindent \emph{Step 3.} Next, we claim that $v^0\not\equiv 0$.  This follows by the same arguments as in Step 4 of the proof of Theorem~\ref{thm:existence1}.  Indeed, assume the contrary, so $m_0=0$. Then by Lemma~\ref{lem:binding} and \eqref{energydecomp} we would have
	\[
		   I_V(M)\leq I_0(M) \leq I_0(m_1) + I_0(M-m_1) = I_V(M),
	\]
and so $I_V(M)=I_0(M)$.  But the energy functional $\fancE_0$ is translation invariant, hence we may pull back the component, $\tld{u}_n(x):= u_n(x+y_n)$ with the same $\fancE_0$ value, and obtain
\begin{align*}
   I_V(M)&= I_0(M)= \lim_{n\to\infty}\fancE_0(\tld{u}_n) = \lim_{n\to\infty}\left[\fancE_V(\tld{u}_n) + \int_{\RRR} V(x)|\tld{u}_n|^2\,dx \right] \\
   &\geq I_V(M) + \liminf_{n\to\infty} \int_{\RRR} V(x)|\tld{u}_n|^2\,dx = I_V(M) + \int_{\RRR} V(x) |v^1|^2\,dx \\
   &> I_V(M),
\end{align*}
a contradiction. Therefore $m_0>0$, and $v^0$ is a nontrivial ground-state of $I_V(m_0)$.

\medskip

\noindent \emph{Step 4.}  Both $v^0$ and $v^1$ are strictly positive and have exponential decay, i.e., $0<v^i(x)\leq C e^{-\nu |x|}$, for constants $C,\, \nu>0$ and for $i=0,1$.
To show this, we first follow the Appendix in Lions \cite{Lions1987} and note that, by Ekeland's variational principle \cite{Ek1979}, we may find a minimizing sequence $\tilde u_n$ for $I_V(M)$, with 
\beqn\label{close}
\|\tilde u_n- u_n\|_{H^1(\RRR)}\to 0
\eeqn
 and which approximately solve the Euler-Lagrange equations, 
$\| D\fancE_V(\tilde u_n) - \mu_n \tilde u_n\|_{H^{-1}(\RRR)}\to 0$, that is,
$$  -\Delta \tilde u_n + \left[ f(\tilde u_n) -V(x) 
    \left(|\tilde u_n|^2*|x|^{-1}\right) - \mu_n\right] \tilde u_n \longrightarrow 0,
$$
in $H^{-1}(\RRR)$, with $f(t)=\frac{5}3 t^{4/3}-\frac{4}3 t^{2/3}$, and Lagrange multiplier $\mu_n$.  As $\|\tilde u_n\|_{H^1(\RRR)}$ is uniformly bounded, using $\tilde u_n$ as a test function we readily show that the Lagrange multipliers $\mu_n$ are bounded, and passing to a limit along a subsequence, $\mu_n\to \mu$.  Furthermore, by Step 2 and \eqref{close}, $\tilde u_n$ admits the same decomposition \eqref{energydecomp} into components $v^i$ as does $u_n$, and using weak convergence we obtain a limiting PDE for each component,
\begin{gather*}
-\Delta v^0 +\left[f(v^0) - V(x) + \left((v^0)^2*|x|^{-1}\right)\right]v^0
  = \mu v^0, \\
-\Delta v^1 +\left[f(v^1) + \left((v^1)^2*|x|^{-1}\right)\right]v^1
  = \mu v^1,
\end{gather*}
with the {\em same Lagrange multiplier} $\mu$.  By minimization, $v^i\geq 0$ and by the strong maximum principle, we may conclude that each $v^i>0$ for $i=0,1$.

Next, we show that the Lagrange multiplier $\mu<0$.
Following the proof of Theorem~1 of Le Bris \cite{LeBris1993}, we define the spherical mean of an integrable $\psi$ as $\bar\psi(x)={1\over 4\pi}\int_{\sigma\in\mathbb{S}^2} \psi(|x|\sigma)\, dS(\sigma)$, and note that by Newton's Theorem (see Theorem 9.7 of Lieb and Loss \cite{LiebLoss97}),
$$ \overline{(v^0)^2*|x|^{-1}} = \overline{(v^0)^2}*|x|^{-1} \leq 
          |x|^{-1}\int_{\RRR} (v^0)^2\,dx \leq {M\over |x|}.
$$
By (H2), there exists $R>0$ for which $\overline{V}(x)\geq {M\over |x|}$ for all $|x|\geq R$, and hence 
\beqn\label{S4.2}  \overline{(v^0)^2*|x|^{-1} - V} \leq 0 \quad\text{for all $|x|\geq R$.}
\eeqn

Assume for a contradiction that $\mu\geq 0$.
Set $W:= \frac53 |v^0|^{4/3} + [(v^0)^2*|x|^{-1}] - V$, so $v^0$ satisfies the differential inequality,
$$  -\Delta v^0 + W v^0 = \left(\frac23 (v^0)^{5/3} + \mu\right) v^0\geq 0  $$
in $\RRR$.  By \eqref{S4.2}, outside $B_R$, $\overline{W}_+= \frac53\overline{(v^0)^{4/3}}\in L^{3/2}$.  Applying Theorem~7.18 of Lieb \cite{Lieb1981}, we conclude that $v^0\not\in L^2(B_R^c)$, a contradiction.  Thus, $\mu<0$.

Finally, from equation (66) of Lions \cite{Lions1987} we may conclude that the solutions are exponentially localized, 
\beqn\label{exp}
		|\nabla v^i(x)|+|v^i(x)| \leq C e^{-\nu |x|}
\eeqn
for $i=0,1$ with $0<\nu < \sqrt{-\mu}$.

\medskip

\noindent \emph{Step 5.}
We are ready to complete the existence argument. Assume, for a contradiction, that $u_n$ is a minimizing sequence for $I_V(M)$ with no convergent subsequence. By Step 2, we obtain $m_i>0$, $v^i\in H^1(\RRR)$ for $i=0,1$ satisfying  \eqref{energydecomp}. Moreover, we claim that
	\beqn \label{eqn:last-claim}
		I_V(m_0+m_1) < I_V(m_0) + I_0(m_1).
	\eeqn
Assuming the claim holds, taking $m=m_0+m_1$ in Lemma \ref{lem:binding}, and using \eqref{energydecomp}, we obtain that
	\beqns
		\begin{aligned}
			I_V(M) &\leq I_V(m_0+m_1) + I_0(M-m_0-m_1) \\
					   &< I_V(m_0) + I_0(m_1) + I_0(M-m_0-m_1) \\
					   &= I_V(M),
		\end{aligned}
	\eeqns
a contradiction. We therefore conclude that $m_0=M$, and the minimizing sequence converges.

In order to prove \eqref{eqn:last-claim} we will construct a family of functions based on the elements obtained in \eqref{energydecomp}: For $t>0$, let
	\[
		   w_t(x)\, :=\,  v^0(x) +  v^1(x- t\xi),
	\]
where $\xi\in\RRR$ with $|\xi|=1$, and define the admissible function
	\[
		\tld{w}_t(x) \,:=\, \frac{\sqrt{m_0+m_1}\,w_t(x)}{\|w_t\|_{L^2(\RRR)}}
	\] 
so that $\int_{\RRR} \tld{w}_t^2\,dx = m_0+m_1$.  
However, by the exponential decay \eqref{exp} we note that 
	\[
		|\fancE_V(\tld{w}_t) - \fancE_V(w_t)| \leq C e^{-\nu t},
	\]
and hence in order to estimate $\fancE_V(\tld{w}_t)$ it suffices to estimate $\fancE_V(w_t)$.

Again using the exponential decay of the component functions $v^i$, $i=0,1$, and arguing as in the proof of Corollary II.2(ii) in Lions \cite{Lions1987}, for $t>0$ large, we obtain the decomposition
\begin{multline}   \fancE_V(w_t) = \fancE_V(v^0) +  I_0(m_i) \\
 +  2 \int_{\RRR}\!\int_{\RRR} {|v^0(x)|^2 |v^1(y-t\xi)|^2\over 4\pi |x-y|} \,dx dy  -  \int_{\RRR} V(x)|v^1(x-t\xi)|^2\, dx + o\left(\frac{1}{t}\right). \nonumber
\end{multline}

Now we show that for large $t>0$, the second line above is strictly negative.
%	\[
%	\sum_{\substack{i,j=1\\ i\neq j}}^N \int_{\RRR}\!\int_{\RRR} {|v^i(x-t\xi_i)|^2 |v^j(y-t\xi_j)|^2\over 4\pi |x-y|} \,dx dy - \sum_{i=1}^N \int_{\RRR} V(x)|v^i(x-t\xi_i)|^2\, dx < 0.
% \]
First, note that
\begin{multline} \nonumber
t \int_{\RRR}\!\int_{\RRR} {|v^0(x)|^2 |v^1(y-t\xi)|^2\over 4\pi |x-y|}\, dx dy = {1\over 4\pi} \int_{\RRR}\!\int_{\RRR} {|v^0(x)|^2 |v^1(y)|^2\over  |\xi - (x-y)/t|} \,dx dy \\  \xrightarrow[{}\ t\to\infty {}\ ]{}  {\| v^0\|_{L^2(\RRR)}^2 \, \| v^1\|_{L^2(\RRR)}^2\over 4\pi|\xi|} = \frac{m_i\,m_j}{4\pi}
\end{multline}
by dominated convergence theorem. That is, this term is $O(t^{-1})$.  

To estimate the other term, first note that \ref{itm:H4} implies that for every $A>0$ there exists $t_0>0$ such that $tV(x) \geq A$ for $|x|=t$ whenever $t \geq t_0$, i.e.,
	\[
		\inf_{|x|=t} V(x) \geq \frac{A}{|x|}
	\]
when $|x|=t \geq t_0$. Next, choose $r_0$ and $C>0$ such that $\int_{B_{r_0}(0)} |v^1|^2\, dx \geq C>0$. Then, for $t>2r_0$ we have that
\begin{align*}
t \int_{\RRR} V(x) |v^1(x-t\xi)|^2\, dx & \geq t\int_{B_{r_0}(0)} V(x+ t\xi) |v^1(x)|^2\, dx \\
&\geq C\, t\, \inf_{x\in B_{r_0}(0)} V(x+ t\xi) \\
&\geq C\, t\, \inf_{t-r_0\leq |x| \leq t+r_0} \frac{A}{|x|} = \frac{C\,t\,A}{t+r_0} \geq  \frac{C\,A}{2}
\end{align*}
for large enough $t>0$.
Since the above holds for all $A>0$ we have that
	\[
		t\int_{\RRR} V(x) |v^1(x-t\xi)|^2\, dx \xrightarrow[{} \  t\to\infty \ {}]{} \infty.
	\] 
In particular, the confinement term dominates the other cross terms for $t>0$ sufficiently large, and thus
	\[
		I_V(m_0+m_1)\leq \fancE_V(\tld{w}_t) <  I_V(m_0) + I_0(m_1),
	\] 
proving our claim \eqref{eqn:last-claim}.
\end{proof}

%\edred{
%\begin{remark}\label{rem:two-split}
%It could be possible to simplify the proofs of Theorems \ref{thm:existence1} and \ref{thm:existence2} by avoiding an iteration of the splitting argument; however, we are not able to show a priori that the first piece (denoted by $w^0$ and $v^0$, respectively) is non-trivial. In other words, without reiterating the splitting arguments we cannot conclude that there is a piece which does not require translation and is a ground-state of $e_V$ or $I_V$, respectively.
%
%Although it is probable that the proof of Theorem \ref{thm:existence2} can proceed in an analogous way to the proof of Theorem \ref{thm:existence1} by using some additional decay estimates proved by Benguria, Br\'{e}zis, and Lieb \cite{BeBrLi1981}, we chose to show the complete splitting here which complement the results of Nam and van den Bosch to the super-Newtonian case, and which can be used in the future to study the non-existence of ground-states of the TFDW model with no external potential \cite{LO1}.
%\end{remark}
%}

\bigskip

\subsection*{Acknowledgments} The authors would like to thank the referees for their invaluable comments, which allowed us to simplify the proof of Theorem \ref{thm:existence2} significantly. SA, LB, and RC were supported by NSERC (Canada) Discovery Grants.

\bibliographystyle{IEEEtranS}
\def\url#1{}
\bibliography{AlBrChTo17_Bib}

\end{document}